\documentclass[12pt]{amsart}

\usepackage{geometry}
\theoremstyle{plain}

\newtheorem{theorem}{Theorem}
\newtheorem{corollary}{Corollary}
\newtheorem{definition}{Definition}
\newtheorem{lemma}{Lemma}

\usepackage{amssymb}
\usepackage{amsmath}
\usepackage{amsthm}
\usepackage{mathtools}
\usepackage{multirow}
\usepackage{thmtools}

\usepackage[colorlinks, citecolor=blue]{hyperref}

\begin{document}

\title[Representations of flat virtual braids by automorphisms of free group]{Representations of flat virtual braids by automorphisms of free group}

\author[B.~Chuzhinov]{Bogdan Chuzhinov}
\address{Novosibirsk State University, 630090 Novosibirsk, Russia}
\email{b.chuzhinov@g.nsu.ru}

 \author[A.~Vesnin]{Andrey Vesnin}
\address{Sobolev Institute of Mathematics, 630090 Novosibirsk, Russia; 
Novosibirsk State University, 630090 Novosibirsk, Russia} 
\email{vesnin@math.nsc.ru}

	\thanks{
A.V.'s work  was carried out in the framework of the State Task to the  Sobolev Institute of Mathematics (Project FWNF-2022-0004)} 
	\date{\today}

\begin{abstract}
Representations of braid group $B_n$ on $n \geq 2$ strands by automorphisms of a free group of rank $n$ go back to Artin (1925). In 1991 Kauffman introduced a theory of virtual braids and virtual knots and links. The virtual braid group $VB_n$ on $n \geq 2$ strands is an extension of the classical braid group $B_n$ by the symmetric group $S_n$. In this paper we consider flat virtual braid groups $FVB_n$ on $n\geq 2$ strands and construct a family of representations of $FVB_n$ by automorphisms of free groups of rank  $2n$. It has been established that these representations do not preserve the forbidden relations between classical and virtual generators. We investigated some algebraic properties of the constructed representations. In particular, we established conditions of faithfulness in case $n=2$ and proved that the kernel contains a free group of rank two for $n\ge3$.
\end{abstract}



\maketitle

\section{Introduction}
The foundations of the braid groups theory were laid down in the works of E.~Artin in the 1920s. In~\cite{A25} he defined the \textit{braid group} $B_n$ on $n \geq 2$ strands as a group with generators $\sigma_1, \ldots, \sigma_{n-1}$ and defining relations:
\begin{align}
	\sigma_{i}\sigma_{j}&=\sigma_{j}\sigma_{i}, &&|i-j|\ge 2, \label{eq1} \\
	\sigma_{i}\sigma_{i+1}\sigma_{i}&=\sigma_{i+1}\sigma_{i}\sigma_{i+1},  &&1\le i\le n-2. \label{eq2}
\end{align}
A set $\{\sigma_1, \ldots, \sigma_{n-1}\}$ is called the  \textit{standard generators}, or the \textit{Artin generators} of the braid group $B_n$. The generator $\sigma_i \in B_n$ and its inverse $\sigma_i^{-1}$ are presented  geometrically in Figure~\ref{fig1}. 
\begin{figure}[h!]
\begin{center}
\unitlength=1.0mm
\begin{picture}(70,23)(0,3)
\thicklines \put(0,0){\begin{picture}(0,30)
\put(-15,10){\circle*{1}}
\put(-15,20){\circle*{1}}
\put(-5,10){\circle*{1}}
\put(-5,20){\circle*{1}}
\put(0,10){\circle*{1}}
\put(0,20){\circle*{1}}
\put(10,10){\circle*{1}}
\put(10,20){\circle*{1}}
\put(15,10){\circle*{1}}
\put(15,20){\circle*{1}}
\put(25,10){\circle*{1}}
\put(25,20){\circle*{1}}
\qbezier(-15,10)(-15,10)(-15,20)
\put(-10,15){\makebox(0,0)[c]{$\cdots$}}
\qbezier(-5,10)(-5,10)(-5,20)
\qbezier(0,10)(0,10)(10,20)
\qbezier(6,14)(6,14)(10,10)
\qbezier(4,16)(4,16)(0,20) 
\qbezier(15,10)(15,10)(15,20)
\put(20,15){\makebox(0,0)[c]{$\cdots$}}
\qbezier(25,10)(25,10)(25,20) 
\put(-15,25){\makebox(0,0)[c]{\small $1$}}
\put(0,25){\makebox(0,0)[c]{\small $i$}}
\put(10,25){\makebox(0,0)[c]{\small $i+1$}}
\put(25,25){\makebox(0,0)[c]{\small $n$}}
\put(5,5){\makebox(0,0)[c]{$\sigma_i$}} \end{picture}}
\put(60,0){\begin{picture}(0,30) 
\thicklines
\put(-15,10){\circle*{1.}}
\put(-15,20){\circle*{1.}}
\put(-5,10){\circle*{1.}}
\put(-5,20){\circle*{1.}}
\put(0,10){\circle*{1.}}
\put(0,20){\circle*{1.}}
\put(10,10){\circle*{1.}}
\put(10,20){\circle*{1.}}
\put(15,10){\circle*{1.}}
\put(15,20){\circle*{1.}}
\put(25,10){\circle*{1.}}
\put(25,20){\circle*{1.}}
\qbezier(-15,10)(-15,10)(-15,20)
\put(-10,15){\makebox(0,0)[c]{$\cdots$}}
\qbezier(-5,10)(-5,10)(-5,20)
 \qbezier(0,10)(0,10)(4,14)
\qbezier(6,16)(6,16)(10,20) \qbezier(10,10)(10,10)(0,20)
\qbezier(15,10)(15,10)(15,20)
\put(20,15){\makebox(0,0)[c]{$\cdots$}}
\qbezier(25,10)(25,10)(25,20) 
\put(-15,25){\makebox(0,0)[c]{\small $1$}}
\put(0,25){\makebox(0,0)[c]{\small $i$}}
\put(10,25){\makebox(0,0)[c]{\small $i+1$}}
\put(25,25){\makebox(0,0)[c]{\small $n$}}
\put(5,5){\makebox(0,0)[c]{$\sigma^{-1}_i$}} \end{picture}}
\end{picture} \caption{Generator $\sigma_i \in B_n$ and its inverse $\sigma_i^{-1}$.} \label{fig1}
\end{center} 
\end{figure}
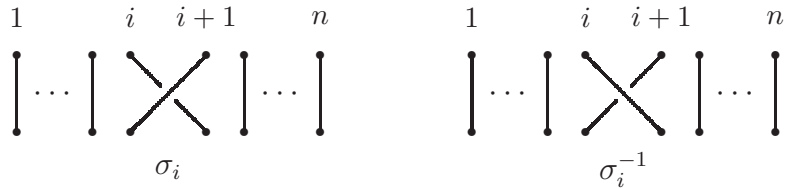

There is very nice relation between braid groups and knot theory  based on \textit{Alexender's theorem} and \textit{Markov's theorem}~\cite{KT}.  Invariants, arising from representations of braid groups, play an important role in classical knot theory and its generalizations.

Artin discovered faithful representation $\varphi_n\colon B_n\to{\rm{Aut(\mathbb F}}_n)$, where ${\mathbb{F}}_n = \langle x_1, \ldots , x_n\rangle$ -- is the free group of rank $n$. Homomorphism $\varphi_n$ maps generator $\sigma_i \in B_n$ to the following automorphism $\varphi_n(\sigma_i)$: 
\begin{align*}
	\varphi_n(\sigma_i):
	\begin{cases}
		x_i \mapsto x_{i+1},\\
		x_{i+1} \mapsto x_{i+1}^{-1}x_ix_{i+1}, \\
		x_j \mapsto x_j, \qquad j \neq i, i+1.
	\end{cases}
\end{align*}
Note, that for each $i$, $1 \leq i \leq n-1$ one has $\varphi_n(\sigma_i) (x_1 \cdots x_n) =  x_1 \cdots x_n$. Therefore $\varphi_n (\beta) (x_1 \cdots x_n) =  ( x_1 \cdots x_n)$ for every $\beta \in B_n$. Moreover, it is shown by Artin~\cite{A25, A47} that an automorphism $g \in \operatorname{Aut} (\mathbb F_n)$ is equal to $\varphi_n(\beta)$ for some $\beta \in B_n$ if and only if the following two conditions are satisfied: (i) $f(x_i)$ is conjugate of some $x_j$ for $i=1, \ldots, n-1$; and (ii) $f(x_1 \cdots x_n) = x_1 \cdots x_n$. 

In what follows, if any automorphism acts on a generator identically, we will not write this action. We write the composition of automorphisms in the order of their application from left to right, namely, $\varphi\psi(f)=\psi(\varphi(f))$.

Virtual braids were introduced by Kauffman in his founding paper~\cite{K99}  together with virtual knots and links, see also~\cite{KL}. In the same paper he defined the virtual braid group $VB_n$ on $n \geq 2$ strands, generated by the elements $\sigma_1,\ldots,\sigma_{n-1}$ similarly to the classical braid group and $\rho_1,\ldots \rho_{n-1} $ that satisfy \emph{skew relations} (\ref{eq1})~-- (\ref{eq2}), \emph{substitution group relations} (\ref{eq3})~-- (\ref{eq5}) and \emph{mixed relations} (\ref{eq6})~-- (\ref{eq7}):
\begin{align}
	\rho_{i}^{2}&=1, &&1\le i\le n-1, \label{eq3} \\
	\rho_{i}\rho_{j}&=\rho_{j}\rho_{i},  &&|i-j|\ge 2, \label{eq4} \\
	\rho_{i}\rho_{i+1}\rho_{i}&=\rho_{i+1}\rho_{i}\rho_{i+1}, &&1\le i\le n-2, \label{eq5} \\
	\rho_{i}\sigma_{j}&=\sigma_{j}\rho_{i},  &&|i-j|\ge 2,\label{eq6} \\
	\rho_{i}\rho_{i+1}\sigma_{i}&=\sigma_{i+1}\rho_{i}\rho_{i+1},  &&1\le i\le n-2. \label{eq7} 
\end{align}

Generator $\rho_i \in VB_n$ is  presented geometrically in Figure~\ref{fig2}.
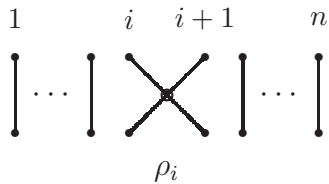
\begin{figure}[h!]
\begin{center}
\unitlength=1.0mm
\begin{picture}(0,25)(5,3)
\thicklines 
\put(-15,10){\circle*{1}}
\put(-15,20){\circle*{1}}
\put(-5,10){\circle*{1}}
\put(-5,20){\circle*{1}}
\put(0,10){\circle*{1}}
\put(0,20){\circle*{1}}
\put(10,10){\circle*{1}}
\put(10,20){\circle*{1}}
\put(15,10){\circle*{1}}
\put(15,20){\circle*{1}}
\put(25,10){\circle*{1}}
\put(25,20){\circle*{1}}
\qbezier(-15,10)(-15,10)(-15,20)
\put(-10,15){\makebox(0,0)[c]{$\cdots$}}
\qbezier(-5,10)(-5,10)(-5,20) 
\qbezier(0,10)(0,10)(10,20) 
\qbezier(10,10)(10,10)(0,20)
\qbezier(15,10)(15,10)(15,20)
\put(20,15){\makebox(0,0)[c]{$\cdots$}}
\qbezier(25,10)(25,10)(25,20) 
\put(5.7,15){\makebox(0,0)[c]{\circle{1.4}}}
\put(-15,25){\makebox(0,0)[c]{\small $1$}}
\put(0,25){\makebox(0,0)[c]{\small $i$}}
\put(10,25){\makebox(0,0)[c]{\small $i+1$}}
\put(25,25){\makebox(0,0)[c]{\small $n$}}
\put(5,5){\makebox(0,0)[c]{$\rho_i$}}
\end{picture} \caption{Generator $\rho_i \in VB_n$.} \label{fig2}
\end{center} 
\end{figure}

Geometric braids corresponding to the mixed relation (\ref{eq7}) are presented in Figure~\ref{fig3}.

\begin{figure}[h!]
\begin{center}
\unitlength=1.0mm
\begin{picture}(60,45)(0,3)
\thicklines 
\put(0,10){\circle*{1}}
\put(10,10){\circle*{1}}
\put(20,10){\circle*{1}}
\put(40,10){\circle*{1}}
\put(50,10){\circle*{1}}
\put(60,10){\circle*{1}}
\put(0,40){\circle*{1}}
\put(10,40){\circle*{1}}
\put(20,40){\circle*{1}}
\put(40,40){\circle*{1}}
\put(50,40){\circle*{1}}
\put(60,40){\circle*{1}}
 \qbezier(0,10)(0,10)(10,20)
\qbezier(0,20)(0,20)(4,16) 
\qbezier(6,14)(10,10)(10,10)
\qbezier(20,10)(20,10)(20,20)
\qbezier(0,20)(0,20)(0,30)
\qbezier(10,20)(10,20)(20,30)
\qbezier(20,20)(20,20)(10,30)
\put(15,25){\circle{1.4}}
\qbezier(0,30)(0,30)(10,40)
\qbezier(10,30)(10,30)(0,40)
\qbezier(20,30)(20,30)(20,40)
\put(5,35){\circle{1.4}}
\put(0,45){\makebox(0,0)[c]{\small $i$}}
\put(10,45){\makebox(0,0)[c]{\small $i+1$}}
\put(20,45){\makebox(0,0)[c]{\small $i+2$}}
\put(10,5){\makebox(0,0)[c]{$\rho_i \rho_{i+1} \sigma_i$}} 
\put(30,25){\makebox(0,0)[c]{$=$}} 
 \qbezier(40,10)(40,10)(40,20)
\qbezier(50,10)(50,10)(60,20) 
\qbezier(60,10)(60,10)(50,20)
\put(55,15){\circle{1.4}}
\qbezier(40,20)(40,20)(50,30)
\qbezier(50,20)(50,20)(40,30)
\qbezier(60,20)(60,20)(60,30)
\put(45,25){\circle{1.4}}
\qbezier(40,30)(40,30)(40,40)
\qbezier(50,40)(50,40)(54,36)
\qbezier(60,30)(60,30)(56,34)
\qbezier(50,30)(50,30)(60,40)
\put(40,45){\makebox(0,0)[c]{\small $i$}}
\put(50,45){\makebox(0,0)[c]{\small $i+1$}}
\put(60,45){\makebox(0,0)[c]{\small $i+2$}}
\put(50,5){\makebox(0,0)[c]{$\sigma_{i+1} \rho_i \rho_{i+1}$}} 
\end{picture} \caption{The mixed relation $\rho_i \rho_{i+1} \sigma_i = \sigma_{i+1} \rho_i \rho_{i+1}$.} \label{fig3}
\end{center} 
\end{figure}
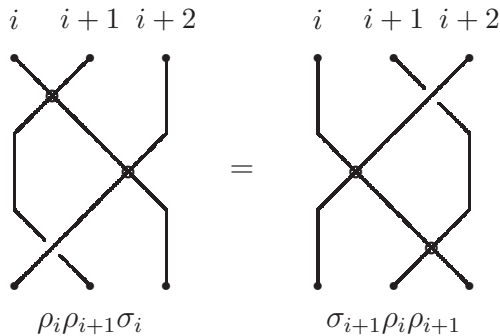

S.~Kamada~\cite{K07} established that the following \emph{Alexander's theorem for virtual braids} holds: If $K$ is a virtual link, then for some $n$ there exists a virtual braid $\beta \in VB_n $ such that $K$ is the closure of $\beta$.

It is known as shown in \cite{GPV} that relations
\begin{eqnarray}
	\rho_{i}\sigma_{i+1}\sigma_{i} =\sigma_{i+1}\sigma_{i}\rho_{i+1}, \qquad \qquad 1\le i\le n-2, \label{eq8} \\
	\rho_{i+1}\sigma_{i}\sigma_{i+1}=\sigma_{i}\sigma_{i+1}\rho_{i}, \qquad \qquad 1\le i\le n-2. \label{eq9}
\end{eqnarray}
do not hold in the $VB_n$ group. These relations (\ref{eq8})~-- (\ref{eq9}) are called \emph{forbidden relations}, see Figures~\ref{fig4} and~\ref{fig5}. The group $WB_n$ is obtained from $VB_n$ by adding the relation~(\ref{eq8}) and is called \emph{welded braid group}~\cite{FRR97}. The same group $WB_n$ is obtained by adding the relation~(\ref{eq9}) to the group $VB_n$. Adding both relations (\ref{eq8}) and (\ref{eq9}) to $VB_n$ leads to unknotting transformations for virtual knots and links~\cite{GPV, K01, K07}. Other unknotting operations for links, virtual links and welded links are given, for example, in~\cite{GPV20, GKPV21, KPV19}. Note that the representations $VB_n \to \rm{Aut} (G_n)$ were constructed, for example, for groups $G_n$ of the following form: $G_n = \mathbb F_n * \mathbb Z^{n+1}$ in~\cite{SW}, $G_n = \mathbb F_n * \mathbb Z^2$ in~\cite{B}, $G_n = \mathbb F_n * \mathbb Z^{2n+1}$ and $G_n = \mathbb F_n * \mathbb Z^n$ in~\cite{BMN}. For structural properties and other representations of the virtual braid groups see~\cite{BP, BEIKNV}. 

\begin{figure}[h!]
\begin{center}
\unitlength=1.0mm
\begin{picture}(60,45)(0,3)
\thicklines 
\put(0,10){\circle*{1}}
\put(10,10){\circle*{1}}
\put(20,10){\circle*{1}}
\put(40,10){\circle*{1}}
\put(50,10){\circle*{1}}
\put(60,10){\circle*{1}}
\put(0,40){\circle*{1}}
\put(10,40){\circle*{1}}
\put(20,40){\circle*{1}}
\put(40,40){\circle*{1}}
\put(50,40){\circle*{1}}
\put(60,40){\circle*{1}}
 \qbezier(0,10)(0,10)(10,20)
\qbezier(0,20)(0,20)(4,16) 
\qbezier(6,14)(10,10)(10,10)
\qbezier(20,10)(20,10)(20,20)
\qbezier(0,20)(0,20)(0,30)
 \qbezier(10,20)(10,20)(20,30)
\qbezier(10,30)(10,30)(14,26) 
\qbezier(16,24)(20,20)(20,20)
\qbezier(0,30)(0,30)(10,40)
\qbezier(10,30)(10,30)(0,40)
\qbezier(20,30)(20,30)(20,40)
\put(5,35){\circle{1.4}}
\put(0,45){\makebox(0,0)[c]{\small $i$}}
\put(10,45){\makebox(0,0)[c]{\small $i+1$}}
\put(20,45){\makebox(0,0)[c]{\small $i+2$}}
\put(10,5){\makebox(0,0)[c]{$\rho_i \sigma_{i+1} \sigma_i $}} 
\put(30,25){\makebox(0,0)[c]{$=$}} 
 \qbezier(40,10)(40,10)(40,20)
\qbezier(50,10)(50,10)(60,20) 
\qbezier(60,10)(60,10)(50,20)
\put(55,15){\circle{1.4}}
\qbezier(60,20)(60,20)(60,30)
\qbezier(40,30)(40,30)(44,26)
\qbezier(50,20)(50,20)(46,24)
\qbezier(40,20)(40,20)(50,30)
\qbezier(40,30)(40,30)(40,40)
\qbezier(50,40)(50,40)(54,36)
\qbezier(60,30)(60,30)(56,34)
\qbezier(50,30)(50,30)(60,40)
\put(40,45){\makebox(0,0)[c]{\small $i$}}
\put(50,45){\makebox(0,0)[c]{\small $i+1$}}
\put(60,45){\makebox(0,0)[c]{\small $i+2$}}
\put(50,5){\makebox(0,0)[c]{$\sigma_{i+1} \sigma_i \rho_{i+1}$}} 
\end{picture} \caption{The forbidden relation $\rho_i \sigma_{i+1} \sigma_i = \sigma_{i+1} \sigma_i \rho_{i+1}$.} \label{fig4}
\end{center} 
\end{figure}
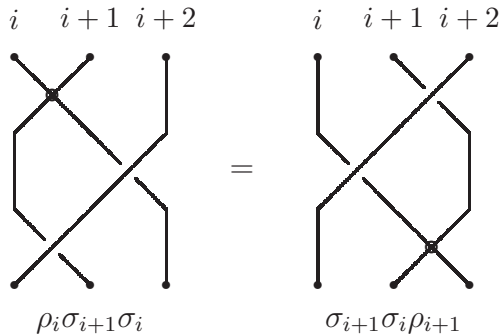

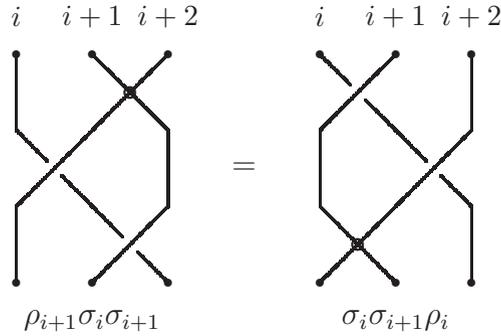
\begin{figure}[h!]
\begin{center}
\unitlength=1.0mm
\begin{picture}(60,43)(0,3)
\thicklines 
\put(0,10){\circle*{1}}
\put(10,10){\circle*{1}}
\put(20,10){\circle*{1}}
\put(40,10){\circle*{1}}
\put(50,10){\circle*{1}}
\put(60,10){\circle*{1}}
\put(0,40){\circle*{1}}
\put(10,40){\circle*{1}}
\put(20,40){\circle*{1}}
\put(40,40){\circle*{1}}
\put(50,40){\circle*{1}}
\put(60,40){\circle*{1}}
\qbezier(0,10)(0,10)(0,20)
 \qbezier(10,10)(10,10)(20,20)
\qbezier(10,20)(10,20)(14,16) 
\qbezier(16,14)(20,10)(20,10)
 \qbezier(0,20)(0,20)(10,30)
\qbezier(0,30)(0,30)(4,26) 
\qbezier(6,24)(10,20)(10,20)
\qbezier(20,20)(20,20)(20,30)
\qbezier(10,30)(10,30)(20,40)
\qbezier(20,30)(20,30)(10,40)
\qbezier(0,30)(0,30)(0,40)
\put(15,35){\circle{1.4}}
\put(0,45){\makebox(0,0)[c]{\small $i$}}
\put(10,45){\makebox(0,0)[c]{\small $i+1$}}
\put(20,45){\makebox(0,0)[c]{\small $i+2$}}
\put(10,5){\makebox(0,0)[c]{$\rho_{i+1} \sigma_{i} \sigma_{i+1}$}} 
\put(30,25){\makebox(0,0)[c]{$=$}} 
 \qbezier(60,10)(60,10)(60,20)
\qbezier(40,10)(40,10)(50,20) 
\qbezier(50,10)(50,10)(40,20)
\put(45,15){\circle{1.4}}
\qbezier(40,20)(40,20)(40,30)
\qbezier(50,30)(50,30)(54,26)
\qbezier(60,20)(60,20)(56,24)
\qbezier(50,20)(50,20)(60,30)
\qbezier(60,30)(60,30)(60,40)
\qbezier(40,40)(40,40)(44,36)
\qbezier(50,30)(50,30)(46,34)
\qbezier(40,30)(40,30)(50,40)
\put(40,45){\makebox(0,0)[c]{\small $i$}}
\put(50,45){\makebox(0,0)[c]{\small $i+1$}}
\put(60,45){\makebox(0,0)[c]{\small $i+2$}}
\put(50,5){\makebox(0,0)[c]{$\sigma_{i} \sigma_{i+1} \rho_{i}$}} 
\end{picture} \caption{The forbidden relation $\rho_{i+1} \sigma_{i} \sigma_{i+1} = \sigma_{i} \sigma_{i+1} \rho_{i}$.} \label{fig5}
\end{center} 
\end{figure}

In~\cite{K00} the \emph{group of flat virtual braids} $FVB_n$ on $n$ strands was introduced as a result of adding the relations~(\ref{eq10}) to the group $VB_n$:
\begin{equation}
	\sigma_{i}^{2} =1, \qquad1\le i\le n-1. \label{eq10}
\end{equation}

We summarize the above discussions in the following definition. 
\begin{definition}
For $n \geq 2$ a group with generators $\sigma_1, \ldots, \sigma_{n-1}$, $\rho_1, \ldots, \rho_{n-1}$ and the following defining relations:  
$$
\begin{array}{ccl}
 \sigma_{i}^{2}=1, &  \rho_{i}^{2}=1,   &1\le i\le n-1,  \\
\sigma_{i} \, \sigma_{i+1}\, \sigma_{i}=\sigma_{i+1} \, \sigma_{i} \, \sigma_{i+1},   & \rho_{i} \, \rho_{i+1} \, \rho_{i}= \rho_{i+1} \, \rho_{i} \, \rho_{i+1}, &  1\le i\le n-2, \\
\sigma_{i}\, \sigma_{j} = \sigma_{j}\, \sigma_{i},   &\rho_{i} \, \rho_{j} = \rho_{j} \, \rho_{i}, & |i-j|\ge 2 
\end{array}
$$
and 
$$
\begin{array}{cl}
\rho_{i} \, \rho_{i+1} \, \sigma_{i} = \sigma_{i+1} \, \rho_{i} \, \rho_{i+1}, & 1\le i\le n-2, \\
\rho_{i} \,\sigma_{j} = \sigma_{j} \, \rho_{i}, & |i-j|\ge 2.
\end{array}
$$
is called the \emph{flat virtual braid group $FVB_n$ on $n$ strands}.
\end{definition}
Generator $\sigma_i \in FVB_n$  is presented geometrically in Figure~\ref{fig6} and generator $\rho \in FVB_n$ is presented  geometrically in Figure~\ref{fig2}. 
\begin{figure}[h!]
\begin{center}
\unitlength=1.0mm
\begin{picture}(40,25)(-15,3)
\thicklines 
\put(-15,10){\circle*{1}}
\put(-15,20){\circle*{1}}
\put(-5,10){\circle*{1}}
\put(-5,20){\circle*{1}}
\put(0,10){\circle*{1}}
\put(0,20){\circle*{1}}
\put(10,10){\circle*{1}}
\put(10,20){\circle*{1}}
\put(15,10){\circle*{1}}
\put(15,20){\circle*{1}}
\put(25,10){\circle*{1}}
\put(25,20){\circle*{1}}
\qbezier(-15,10)(-15,10)(-15,20)
\put(-10,15){\makebox(0,0)[c]{$\cdots$}}
\qbezier(-5,10)(-5,10)(-5,20) 
\qbezier(0,10)(0,10)(10,20) 
\qbezier(10,10)(10,10)(0,20)
\qbezier(15,10)(15,10)(15,20)
\put(20,15){\makebox(0,0)[c]{$\cdots$}}
\qbezier(25,10)(25,10)(25,20) 
\put(-15,25){\makebox(0,0)[c]{\small $1$}}
\put(0,25){\makebox(0,0)[c]{\small $i$}}
\put(10,25){\makebox(0,0)[c]{\small $i+1$}}
\put(25,25){\makebox(0,0)[c]{\small $n$}}
\put(5,5){\makebox(0,0)[c]{$\sigma_i$}}
\end{picture} \caption{Generator $\sigma_i \in FVB_n$.} \label{fig6}
\end{center} 
\end{figure}
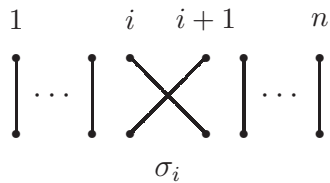

In~\cite{FIKM} the following problem was formulated: Does it exist a representation of the $FVB_n$ group by automorphisms of some group for which the forbidden relations would not hold? In~\cite{B23} such representation $\eta_{n} \colon FVB_n \to {\rm Aut} ( \mathbb F_{2n})$ was constructed, here ${\mathbb{F}}_{2n}= \langle x_1,\ldots, x_n,y_1,\ldots,y_n\rangle$ is a free group of rank $2n$. The homomorphism $\eta_n$ maps generators $\sigma_i, \rho_i \in FVB_n$, where $i = 1, \ldots, n-1$, to the following automorphisms:
\begin{align}
	\eta_n(\sigma_i):
	\begin{cases}
		x_i \mapsto x_{i+1}y_{i+1},\\
		x_{i+1} \mapsto x_iy_{i+1}^{-1},
	\end{cases}
	\quad
	\eta_n(\rho_i):
	\begin{cases}
		x_i \mapsto x_{i+1},\\
		x_{i+1} \mapsto x_i,\\
		y_i \mapsto y_{i+1}, \\ 
		y_{i+1} \mapsto y_i. 
	\end{cases} \label{eq11}
\end{align}
It was shown in~\cite{B23} that the representation $\eta_2 : FVB_2\to {\rm{Aut(\mathbb F}}_4)$ is faithful.

In this paper we construct a family of representations of the $FVB_n$ group by automorphisms of the free group $\mathbb F_{2n}= \langle x_1, \ldots, x_n, y_1,\ldots, y_n \rangle$, which generalize the representation~(\ref{eq11}). Namely, we consider a family of homomorphisms $\Theta_n : FVB_{n} \to {\rm{Aut(\mathbb F}}_{2n})$, which are given by mapping generators $\sigma_i, \rho_i \in FVB_n$, where $i = 1, \ldots, n-1$, to the following automorphisms:
\begin{equation}
	\Theta_n(\sigma_i):
	\begin{cases}
		x_i \mapsto x_{i+1} \, a_i (y_i, y_{i+1}),\\
		x_{i+1} \mapsto x_i \, b_i (y_i, y_{i+1}),
	\end{cases}
	\quad
	\Theta_n(\rho_i):
	\begin{cases}
		x_i \mapsto x_{i+1} \, c_i (y_i, y_{i+1}),\\
		x_{i+1} \mapsto x_{i} \, d_i (y_i, y_{i+1}),\\
		y_i \mapsto y_{i+1}, \\ 
		y_{i+1} \mapsto y_{i},
	\end{cases} \label{eq12.old}
\end{equation}
where the elements $a_i (y_i, y_{i+1})$, $b_i (y_i, y_{i+1})$, $c_i( y_i, y_{i+1})$ and $d_i(y_i, y_{i+1})$ are words in a free group of rank two with generators $\{y_i, y_{i+1} \}$ for each $i = 1, \ldots, n-1$. Thus, the homomorphisms $\Theta_n$ depend only on the choice of the words $a_i, b_i, c_i, d_i$, which define the locally nontrivial action of the automorphism corresponding to the generator of the group $FVB_n$, and in this sense the homomorphisms $\Theta_n$ are \emph{local homomorphisms}.

\smallskip 

The article has the following structure. In Theorem~\ref{theorem1.1} we establish for which $a_i$, $b_i$, $c_i$ and $d_i$ there exists a local homomorphism $\Theta_n$ of the group $FVB_n$ into the automorphism group of the free group $\mathbb F_{ 2n}$. In Section~\ref{section2} we obtain results about the structure of the kernel of the homomorphism $\Theta_n$, in particular, in Theorem~\ref{theorem2.5} we describe the kernel of this homomorphism for $n=2$. In Theorem~\ref{theorem3.2} it will be established that for $n \geq 3$ the kernel of the homomorphism $\Theta_n$ contains a free group of rank 2. We note it was shown earlier in~\cite{B23} that for $ n \geq 3$ the kernel of the homomorphism $\eta_n$, which is a special case of $\Theta_n$, contains an infinite cyclic group.

\section{Existence of local representations} \label{section1}
Let $\mathbb F_{2n}$ be a free group of rank $2n$ with free generators $x_1, \ldots, x_n$, $y_1,\ldots, y_n$.

\begin{theorem} \label{theorem1.1}
	Let $n \ge 2$ and $a_i (y_i, y_{i+1})$, $b_i (y_i, y_{i+1})$, $c_i(y_i, y_{i+1})$, $d_i(y_i, y_{i+1})$ are words in a free group of rank two with generators $\{y_i, y_{i+1} \}$, where $1\le i \le n-1$. Define the map $\Theta_n : FVB_{n} \to {\rm{Aut(\mathbb F}}_{2n})$ by mapping $\sigma_i$ and $\rho_i$ to automorphisms:
	\begin{equation}
		\Theta_n(\sigma_i):
		\begin{cases}
			x_i \mapsto x_{i+1} \, a_i (y_i, y_{i+1}),\\
			x_{i+1} \mapsto x_i \, b_i (y_i, y_{i+1}),
		\end{cases}
		\quad
		\Theta_n(\rho_i):
		\begin{cases}
			x_i \mapsto x_{i+1} \, c_i (y_i, y_{i+1}),\\
			x_{i+1} \mapsto x_{i} \, d_i (y_i, y_{i+1}),\\
			y_i \mapsto y_{i+1}, \\ 
			y_{i+1} \mapsto y_{i}. 
		\end{cases} \label{eq12}
	\end{equation}
	Then the following properties hold.
	\begin{itemize}
		\item[(1)] The map $\Theta_n$ is homomorphism iff
		$$
		b_i(y_i, y_{i+1}) = a_i^{-1}(y_i, y_{i+1}),\qquad 
		c_i(y_i, y_{i+1}) = y_{i+1}^{m_i},\qquad 
		d_i(y_i, y_{i+1}) = y_i^{-m_i},
		$$
		where $m_i\in\mathbb{Z}$ for $1\le i\le n-1$, and
		\begin{align*}
			a_{j}(y_j, y_{j+1}) &= y_{j+1}^{m_j} \, a_{j-1}(y_{j}, y_{j+1}) \, y_j^{-m_{j-1}}, \quad 2\le j\le n-1,
		\end{align*}
		with $n\ge 3$, where $a_1 = w (y_1, y_2)$ for some word $w(A, B)\in\mathbb{F}_2 = \langle A, B \rangle$. 
		\item[(2)] The map $\Theta_n$ does not preserve the forbidden relations. 
	\end{itemize}
\end{theorem}
\begin{proof}
	(1) Let us verify that the relations (\ref{eq1})~-- (\ref{eq7}) and the relation (\ref{eq10}) are preserved under the map $\Theta_{n}$. 
	Denote $s_{i}=\Theta_{n}(\sigma_{i}) \in {\rm Aut} (\mathbb F_{2n})$ and $r_{i}=\Theta_{n}(\rho_{i}) \in {\rm Aut} (\mathbb F_{2n})$. 
	
	The relations (\ref{eq1}), (\ref{eq4}) and (\ref{eq6}) are preserved because $s_i$ acts non-trivially only on $x_i$ and $ x_{i+1}$, while $r_i$ acts non-trivially only on $x_i$, $x_{i+1}$, $y_i$ and $y_{i+1}$.
	
	Since 
	$$
	s_i^2 :  
	\begin{cases}
		x_i  \mapsto  x_{i+1} \, a_i (y_i, y_{i+1})  \mapsto  x_i \, b_i (y_i, y_{i+1}) \, a_i (y_i, y_{i+1}), \\
		x_{i+1}  \mapsto  x_i \, b_i (y_i, y_{i+1})  \mapsto  x_{i+1} \, a_i (y_i, y_{i+1}) \, b_i (y_i, y_{i+1}) , 
	\end{cases}
	$$
	the relation (\ref{eq10}) is preserved if and only if  $b_i (y_i, y_{i+1}) = a_i^{-1} (y_i, y_{i+1})$ for all $1\le i \le n-1$.
	Further,
	$$
	r_i^2 :  
	\begin{cases}
		x_i  \mapsto  x_{i+1} \, c_i (y_i, y_{i+1})  \mapsto  x_i \, d_i (y_i, y_{i+1}) \, c_i (y_{i+1}, y_i), \\
		x_{i+1}  \mapsto  x_i \, d_i (y_i, y_{i+1})  \mapsto  x_{i+1} \, c_i (y_i, y_{i+1}) \, d_i (y_{i+1}, y_i) ,
	\end{cases}
	$$
	and relation (\ref{eq3}) is preserved iff
	\begin{align}\label{d_and_c}
		d_i(y_i, y_{i+1})=c_i^{-1}(y_{i+1},y_i), \qquad 1\le i \le n-1.
	\end{align}
	Consider the actions of automorphisms $r_{i} r_{i+1} s_{i}$ and $s_{i+1} r_{i} r_{i+1}$. We have
	$$
	r_{i} r_{i+1} s_{i} :
	\begin{cases}
		x_{i}  \mapsto  x_{i+1}c_i(y_i,y_{i+1})  \mapsto  x_{i+2}c_{i+1}(y_{i+1},y_{i+2})c_i(y_i,y_{i+2}), \\
		x_{i+1}  \mapsto  x_ic_i^{-1}(y_{i+1},y_i)  \mapsto  x_ic_i^{-1}(y_{i+2},y_i)  \mapsto  x_{i+1}a_i(y_i,y_{i+1})c_i^{-1}(y_{i+2},y_i),  \\ 
		x_{i+2}  \mapsto  x_{i+2}  \mapsto  x_{i+1}c_{i+1}^{-1}(y_{i+2},y_{i+1})  \mapsto  x_ia_i^{-1}(y_i, y_{i+1})c_{i+1}^{-1}(y_{i+2},y_{i+1}),
	\end{cases}
	$$
	and
	$$
	s_{i+1} r_{i} r_{i+1} :
	\begin{cases}
		x_{i}  \mapsto  x_{i+1}c_i(y_i,y_{i+1}) \mapsto x_{i+2}c_{i+1}(y_{i+1},y_{i+2})c_i(y_i,y_{i+2}),\\
		\begin{array}{c}
			x_{i+1}\mapsto x_{i+2}a_{i+1}(y_{i+1},y_{i+2})\mapsto x_{i+2}a_{i+1}(y_i,y_{i+2}) \\ \mapsto x_{i+1}c_{i+1}^{-1}(y_{i+2},y_{i+1})a_{i+1}(y_i,y_{i+1}), 
		\end{array}  \\
		\begin{array}{c}
			x_{i+2} \mapsto x_{i+1}a_{i+1}^{-1}(y_{i+1},y_{i+2}) \mapsto x_ic_i^{-1}(y_{i+1},y_i)a_{i+1}^{-1}(y_i,y_{i+2}) \\ \mapsto x_ic_i^{-1}(y_{i+2},y_i)a_{i+1}^{-1}(y_i,y_{i+1}).
		\end{array}
	\end{cases}
	$$
	Thus, to fulfill the relation (\ref{eq7}), it is necessary and sufficient that 
	\begin{align}\label{a_and_c}
		a_{i+1}(y_i,y_{i+1})c_i(y_{i+2},y_i)=c_{i+1}(y_{i+2},y_{i+1})a_i(y_i, y_{i+1}),
	\end{align}
	holds for all $1\le i\le n-2$.
	
	A similar consideration of the relations (\ref{eq5}) leads to the equalities 
	\begin{align}\label{eqT1.1}
		c_{i+1}(y_i,y_{i+2})c_i(y_{i+1},y_{i+2})=c_{i+1}(y_{i+1},y_{i+2})c_i(y_i,y_{i+2}),
	\end{align}
	for all $1\le i\le n-2$. This is only possible if $c_i(y_i,y_{i+1})=y_{i+1}^{m_i}$ for some $m_i\in \mathbb{Z}$ with $1\le i\le n-1$. Using (\ref{a_and_c}) and (\ref{eqT1.1}) we get 
	\begin{align*}
		a_{i+1}(y_i,y_{i+1})=y_{i+1}^{m_{i+1}}a_i(y_i,y_{i+1})y_i^{-m_i}, \qquad 1\le i\le n-2.
	\end{align*}	
	The fulfillment of the relations (\ref{eq2}) is checked directly. 
	
	(2) Let us now show that the forbidden relations do not hold under the map $\Theta_n$.
	Indeed, we have: 
	$$
	\begin{gathered}
		y_{i} \overset{r_{i}} {\longmapsto}y_{i+1} \overset{s_{i+1}}{\longmapsto} y_{i+1} \overset{s_{i}}{\longmapsto} y_{i+1},\\
		y_{i} \overset{s_{i+1}}{\longmapsto} y_{i} \overset{s_{i}}{\longmapsto} y_{i} \overset{r_{i+1}}{\longmapsto} y_{i},
	\end{gathered}
	$$
	therefore,  $r_{i}s_{i+1}s_{i}\ne s_{i+1}s_{i}r_{i+1}$. Similarly
	$$
	\begin{gathered}
		y_{i} \overset{s_{i}}{\longmapsto} y_{i} \overset{s_{i+1}}{\longmapsto} y_{i} \overset{r_{i}}{\longmapsto} y_{i+1},\\
		y_{i} \overset{r_{i+1}}{\longmapsto} y_{i} \overset{s_{i}}{\longmapsto} y_{i} \overset{s_{i+1}}{\longmapsto} y_{i},
	\end{gathered}
	$$
	therefore, $s_{i}s_{i+1}r_{i}\ne r_{i+1}s_{i}s_{i+1}$.
\end{proof}

Thus the representation $\Theta_n$ given by the formula~(\ref{eq12}) depends on the word $a_1(A, B)=w(A, B) \in \mathbb F_2 = \langle A, B \rangle $, into which we substitute $y_i$ and $y_{i+1}$ instead of $A$ and $B$ respectively, and a set of integers $m=(m_1, \ldots, m_{n-1})$. To emphasize this dependence of the representation on $w$ and $m$, we denote it $\Theta_n^{w,m}$:
\begin{equation}
	\Theta_n^{w,m}(\sigma_i):
	\begin{cases} \displaystyle 
		x_i \mapsto x_{i+1}\,  \prod_{k=i}^{2} y_{i+1}^{m_k} \, w(y_i, y_{i+1})  \, \prod_{k=1}^{i-1} y_i^{-m_k},\\ 
		\displaystyle 
		x_{i+1} \mapsto x_i  \, \prod_{k=i-1}^{1}  y_i^{m_k} \,  (w(y_i, y_{i+1}))^{-1} \, \prod_{k=2}^{i} y_{i+1}^{-m_k},
	\end{cases} \label{homodef1}
\end{equation}
	and
\begin{equation}
	\Theta_n^{w,m}(\rho_i):
	\begin{cases}
		x_i \mapsto x_{i+1} y_{i+1}^{m_i},\\
		x_{i+1} \mapsto x_{i} y_i^{-m_i},\\
		y_i \mapsto y_{i+1}, \\ 
		y_{i+1} \mapsto y_{i},
	\end{cases} \label{homodef2}
\end{equation}
where in the products $\displaystyle \prod_{k=i}^2$ and $\displaystyle \prod_{k=i-1}^1$ it is assumed that $i\geq 2$ and the indices are decreasing, and in the products $\displaystyle \prod_{k=1}^{i-1}$ and $\displaystyle \prod_{k=2}^{i}$ it is assumed $i \geq 2$ and the indices are increasing.

The word $w$ is called the \emph{defining word} for the homomorphism $\Theta_n^{w,m}$. In the particular case when $m_i=0$ for all $i=1, \ldots, n-1$, we will write $\Theta_n^w : FVB_{n} \to {\rm Aut}(\mathbb F_{2n})$ assuming that 
\begin{equation}
	\Theta_n^w(\sigma_i):
	\begin{cases}
		x_i \mapsto x_{i+1} w(y_i, y_{i+1}),\\
		x_{i+1} \mapsto x_i (w(y_i, y_{i+1}))^{-1},
	\end{cases}\quad
	\Theta_n^w(\rho_i):
	\begin{cases}
		x_i \mapsto x_{i+1},\\
		x_{i+1} \mapsto x_{i},\\
		y_i \mapsto y_{i+1}, \\ 
		y_{i+1} \mapsto y_{i}. 
	\end{cases}
\end{equation}
Note that the homomorphism (\ref{eq11}) constructed in~\cite{B23} can be represented as $\eta_{n}= \Theta^w_{n}$ for $w(A,B) =B$.

Let $\beta \in FVB_n$ and $x \in {\mathbb{F}}_{2n}$. Further, to simplify the notation, by $\beta(x)$ we mean $\Theta^{w,m}_n(\beta) (x)$, where the word $w$ and the set $m$ are assumed to be clear from the context.

\section{The kernel of homomorphism $\Theta_n^{w,m}$ and $FVK_n$ group} \label{section2}

In this section we show that the kernel of the homomorphism $\Theta_{n}^{w,m}$ lies in the intersection of the group of flat virtual pure braids and the group of flat virtual kure braids group $FVK_n$ defined below.

Consider the subgroup ${\rm S}_n =\langle\sigma_1\ldots\sigma_{n-1}\rangle$ of $FVB_n$, which is isomorphic to the permutation group of an $n$-element set. The map $\pi_n : FVB_n \to {\rm{S}}_n$ defined on the generators $\sigma_i, \rho_i$ according to the rule:
\begin{align*}
	\pi_n(\sigma_{i})&=\sigma_{i}, &1\le i\le n-1, \\
	\pi_n(\rho_{i})&=\sigma_{i}, &1\le i\le n-1,
\end{align*}
is obviously a homomorphism. 

\begin{definition}
Denote $FVP_n = \operatorname{Ker} (\pi_n)$ and call it  \emph{flat virtual pure braid group on $n$ strands}.
\end{definition}

Similarly, the subgroup $S'_n=\langle\rho_1\ldots\rho_{n-1}\rangle$ of $FVB_n$ is isomorphic to the permutation group of an $n$-element set, and the map $\nu_n : FVB_n\to {\rm {S'_n}}$ defined on generators $\sigma_{i}$, $\rho_{i}$ as follows:
\begin{align*}
	\nu_n(\sigma_{i}^{})&=1, &1\le i\le n-1, \\
	\nu_n(\rho_{i}^{})&=\rho_{i}, &1\le i\le n-1,
\end{align*}
is also a homomorphism.

\begin{definition}
 Denote $FVK_n = \operatorname{Ker} (\nu_n)$ and call it \emph{flat virtual kure braid group on $n$ strands}.
  \end{definition}
 
Here we use the term \emph{flat virtual kure braid} since the term \emph{kure virtual braid group} was used in~\cite{BPT} for kernel of the map $\pi_K : VB_n \to S_n$ which is defined analogously to $\nu_n : FVB_n \to S'$. The group $FVK_n = \operatorname{Ker} (\nu_n)$ also was denoted by  $FH_n$ in \cite{B23} since is the flat analog of the Rabenda's group $H_n$ from~\cite[Prop.~17]{BBD}. 

\begin{lemma} \cite[Prop.~4]{B23}
	The group $FVK_n$ admits a presentation with generators $x_{i,j}$, $1 \leq i \neq j \leq n$ and defining relations
	\begin{equation}
		x_{i,j}^2 = 1, \qquad
		x_{i,j} \, x_{k,l} = x_{k,l} \, x_{i,j}, \qquad
		x_{i,k} \, x_{k,j} \, x_{i,k} = x_{k,j} \, x_{i,k} \, x_{k.j},
	\end{equation}
	where different letters stand for different indices.
\end{lemma}

\begin{corollary}
	The group $FVK_n$ is a Coxeter group with defining relations
	\begin{equation}
		x_{i,j}^2 = 1, \qquad
		(x_{i,j} \, x_{k,l})^2 = 1, \qquad
		(x_{i,k} \, x_{k,j} )^3 =1.
	\end{equation}
\end{corollary}

The following property is a generalization of the property established in \cite[Prop.~9]{B23} for the word $w (A, B) = B$.

\begin{lemma} \label{lemma3.1}
	Let $n \geq 2$. For any word $w \in{\mathbb F_2}$ and any set of integers $m = (m_1, \ldots, m_{n-1})$ $\operatorname{Ker}(\Theta_n^{w,m }) \le FVP_n \cap FVK_n$.
\end{lemma}

\begin{proof}
	Let $g \in \operatorname{Ker}(\Theta_n^{w,m})$. Then $y_i=g(y_i)=\nu_n(g)(y_i)$, since all $\sigma_i$ act identically on $y_i$. But then $\nu_n(g)$ is the identity permutation of the set $\{y_1, \ldots, y_n\}$, which by definition means $g\in FVK_n$.
	
	Next, we show that $g\in FVP_n$. Denote by $G$ the normal closure of the subgroup $\langle y_1,\ldots, y_n\rangle$ in ${\mathbb F}_{2n}$. It is clear that $G$ is a $\Theta_n^{w,m}(FVB_n)$--invariant subgroup. Then $\Theta_n^{w,m}$ induces a homomorphism $\Psi_n^{w,m}\colon FVB_n\to {\rm Aut}({\mathbb F}_{2n}/ G)={\rm Aut }({\mathbb F}_n)$, where ${\mathbb F}_n=\langle x_1,\ldots,x_n\rangle$. From the formulas (\ref{eq12}) we can write out the action of $\Psi_n^{w,m}$ on the generators of the group $FVB_n$:
	\begin{align*}
		\Psi_n^{w,m}(\sigma_i)\colon \begin{cases}
			x_i\mapsto x_{i+1},\\
			x_{i+1}\mapsto x_i,
		\end{cases}\qquad
		\Psi_n^{w,m}(\rho_i)\colon \begin{cases}
			x_i\mapsto x_{i+1},\\
			x_{i+1}\mapsto x_i,
		\end{cases}\qquad
		1\le i \le n-1.
	\end{align*}
	Now it is easy to see that the image of $FVB_n$ under the map $\Psi_n^{w,m}$ is a permutation of the set $\{x_1,\ldots,x_n\}$. It remains to note that if $g\in \operatorname{Ker}(\Theta_n^{w,m})$, then $g \in \operatorname{Ker}(\Psi_n^{w,m})=FVP_n$.
\end{proof}

Since $S'_n\le FVB_n$, then the decomposition of $FVB_n = FVK_n \rtimes S'_n$ follows directly from the definition of $FVK_n$. Considering the restriction of the homomorphism $\pi_n$ to $FVK_n$, we obtain the homomorphism $\xi_n\colon FVK_n\to S_n$. Note that its kernel is $X_n= \operatorname{Ker}(\xi_n)=FVP_n \cap FVK_n$. Further, since $S'_n\le FVK_n$, we obtain the decomposition  $FVK_n=X_n \rtimes S_n$. Thus, we have the following decomposition of the group of flat virtual braids:
$$
FVB_n=(X_n\rtimes S_n)\rtimes S'_n.
$$

As it invented in~\cite{BBD}, we denote
\begin{align*}
	&\lambda_{i,i+1}=\rho_i\sigma_i, &1\le i\le n-1,\\
	&\lambda_{i,j}=\rho_{j-1}\rho_{j-2}\ldots\rho_{i+1}\lambda_{i,i+1}\rho_{i+1}\ldots \rho_{j-2}\rho_{j-1}, &j-i\ge 2.
\end{align*}
Element $\lambda_{i,j}$ is presented geometrically in Figure~\ref{fig7}, 
\begin{figure}[ht]
\begin{center}
\unitlength=1.0mm
\begin{picture}(60,80)(0,3)
\thicklines 
\put(0,75){\makebox(0,0)[c]{\small $i$}}
\put(10,75){\makebox(0,0)[c]{\small $i+1$}}
\put(20,75){\makebox(0,0)[c]{\small $i+2$}}
\put(40,75){\makebox(0,0)[c]{\small $j-2$}}
\put(50,75){\makebox(0,0)[c]{\small $j-1$}}
\put(60,75){\makebox(0,0)[c]{\small $j$}}
\qbezier(0,10)(0,10)(0,15)
\qbezier(10,10)(10,10)(10,15)
\qbezier(20,10)(20,10)(20,15)
\qbezier(40,10)(40,10)(40,15)
\qbezier(50,10)(50,10)(60,15)
\qbezier(50,15)(50,15)(60,10)
\put(55,12.5){\circle{1.4}}
\qbezier(0,15)(0,10)(0,20)
\qbezier(10,15)(10,15)(10,20)
\qbezier(20,15)(20,15)(20,20)
\qbezier(40,15)(40,15)(50,20)
\qbezier(40,20)(40,20)(50,15)
\qbezier(60,15)(60,15)(60,20)
\put(45,17.5){\circle{1,4}}
\put(0,25){\makebox(0,0)[cc]{$\vdots$}}
\put(10,25){\makebox(0,0)[cc]{$\vdots$}}
\put(20,25){\makebox(0,0)[cc]{$\vdots$}}
\put(40,25){\makebox(0,0)[cc]{$\vdots$}}
\put(50,25){\makebox(0,0)[cc]{$\vdots$}}
\put(60,25){\makebox(0,0)[cc]{$\vdots$}}
\qbezier(0,30)(0,30)(0,35)
\qbezier(10,30)(10,30)(20,35)
\qbezier(10,35)(10,35)(20,30)
\qbezier(40,30)(40,30)(40,35)
\qbezier(50,30)(50,30)(50,35)
\qbezier(60,30)(60,30)(60,35)
\put(15,32.5){\circle{1,4}}
\qbezier(0,35)(0,35)(10,40)
\qbezier(10,35)(10,35)(0,40)
\qbezier(20,35)(20,35)(20,40)
\qbezier(40,35)(40,35)(40,40)
\qbezier(50,35)(50,35)(50,40)
\qbezier(60,35)(60,35)(60,40)
\qbezier(0,40)(0,40)(10,45)
\qbezier(10,40)(10,40)(0,45)
\qbezier(20,40)(20,40)(20,45)
\qbezier(40,40)(40,40)(40,45)
\qbezier(50,40)(50,40)(50,45)
\qbezier(60,40)(60,40)(60,45)
\put(5,42.5){\circle{1,4}}
\qbezier(0,45)(0,45)(0,50)
\qbezier(10,45)(10,45)(20,50)
\qbezier(10,50)(10,50)(20,45)
\qbezier(40,45)(40,45)(40,50)
\qbezier(50,45)(50,45)(50,50)
\qbezier(60,45)(60,45)(60,50)
\put(15,47.5){\circle{1,4}}
\put(0,55){\makebox(0,0)[cc]{$\vdots$}}
\put(10,55){\makebox(0,0)[cc]{$\vdots$}}
\put(20,55){\makebox(0,0)[cc]{$\vdots$}}
\put(40,55){\makebox(0,0)[cc]{$\vdots$}}
\put(50,55){\makebox(0,0)[cc]{$\vdots$}}
\put(60,55){\makebox(0,0)[cc]{$\vdots$}}
\qbezier(0,60)(0,60)(0,65)
\qbezier(10,60)(10,60)(10,65)
\qbezier(20,60)(20,60)(20,65)
\qbezier(40,60)(40,60)(50,65)
\qbezier(40,65)(40,65)(50,60)
\qbezier(60,60)(60,60)(60,65)
\put(45,62.5){\circle{1,4}}
\qbezier(0,65)(0,65)(0,70)
\qbezier(10,65)(10,65)(10,70)
\qbezier(20,65)(20,65)(20,70)
\qbezier(40,65)(40,65)(40,70)
\qbezier(50,65)(50,65)(60,70)
\qbezier(50,70)(50,70)(60,65)
\put(55,67.5){\circle{1.4}}
\put(0,10){\circle*{1}}
\put(10,10){\circle*{1}}
\put(20,10){\circle*{1}}
\put(40,10){\circle*{1}}
\put(50,10){\circle*{1}}
\put(60,10){\circle*{1}}
\put(0,70){\circle*{1}}
\put(10,70){\circle*{1}}
\put(20,70){\circle*{1}}
\put(40,70){\circle*{1}}
\put(50,70){\circle*{1}}
\put(60,70){\circle*{1}}
\put(30,40){\makebox(0,0)[c]{$\cdots$}}
\end{picture} \caption{Element $\lambda_{i,j} \in FVB_n$.} \label{fig7}
\end{center} 
\end{figure}
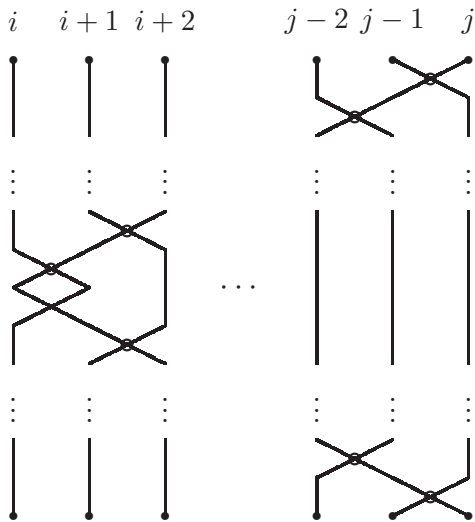

\begin{lemma} \cite{BBD}
	The group $FVP_n$ is generated by the elements $\lambda_{i,j},~1\le i<j\le n$ and the defining relations are:
	\begin{align}
		\lambda_{i,j}\lambda_{k,l}&=\lambda_{k,l}\lambda_{i,j}, \label{eq101} \\
		\lambda_{k,i}\lambda_{k,j}\lambda_{i,j}&=\lambda_{i,j}\lambda_{k,j}\lambda_{k,i}, \label{eq102}
	\end{align}
	where $i,j,k,l$ correspond to different indices.
\end{lemma}

Let us consider the case $n=3$ in more details. As shown in~\cite{BBD},
\begin{align}
	{\rm{FVP}}_3=\langle a,b,c\mid[a,c]=1\rangle, \label{three_gen_FVP}
\end{align}
where $a=\lambda_{2,3}\lambda_{1,3}$, $b=\lambda_{2,3}$ and $c=\lambda_{2,3}^{}\lambda_{1, 2}^{-1}$. These elements presented geometrically in Figure~\ref{fig8}.

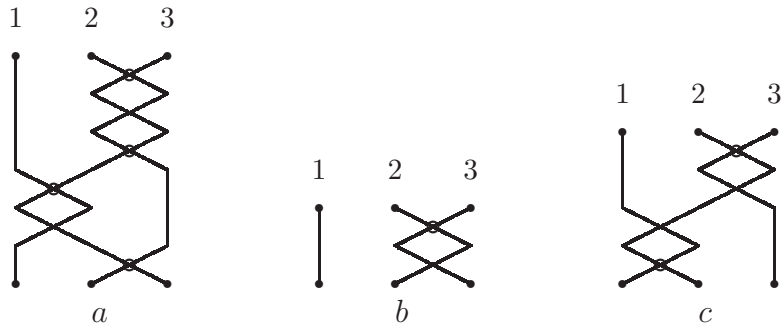
\begin{figure}[ht]
\begin{center}
\unitlength=1.0mm
\begin{picture}(100,43)(0,3)
\thicklines 
\put(0,45){\makebox(0,0)[c]{\small $1$}}
\put(10,45){\makebox(0,0)[c]{\small $2$}}
\put(20,45){\makebox(0,0)[c]{\small $3$}}
\qbezier(0,10)(0,10)(0,15)
\qbezier(10,10)(10,10)(20,15)
\qbezier(20,10)(20,10)(10,15)
\put(15,12.5){\circle{1,4}}
\qbezier(0,15)(0,15)(10,20)
\qbezier(10,15)(10,15)(0,20)
\qbezier(20,15)(20,15)(20,20)
\qbezier(0,20)(0,20)(10,25)
\qbezier(10,20)(10,20)(0,25)
\qbezier(20,20)(20,20)(20,25)
\put(5,22.5){\circle{1,4}}
\qbezier(0,25)(0,25)(0,30)
\qbezier(10,25)(10,25)(20,30)
\qbezier(20,25)(20,25)(10,30)
\put(15,27.5){\circle{1,4}}
\qbezier(0,30)(0,30)(0,35)
\qbezier(10,30)(10,30)(20,35)
\qbezier(20,30)(20,30)(10,35)
\qbezier(0,35)(0,35)(0,40)
\qbezier(10,35)(10,35)(20,40)
\qbezier(20,35)(20,35)(10,40)
\put(15,37.5){\circle{1,4}}
\put(10,5){$a$}
\put(40,25){\makebox(0,0)[c]{\small $1$}}
\put(50,25){\makebox(0,0)[c]{\small $2$}}
\put(60,25){\makebox(0,0)[c]{\small $3$}}
\qbezier(40,10)(40,10)(40,15)
\qbezier(50,10)(50,10)(60,15)
\qbezier(60,10)(60,10)(50,15)
\qbezier(40,15)(40,15)(40,20)
\qbezier(50,15)(50,15)(60,20)
\qbezier(60,15)(60,15)(50,20)
\put(55,17.5){\circle{1,4}}
\put(50,5){$b$}
\put(80,35){\makebox(0,0)[c]{\small $1$}}
\put(90,35){\makebox(0,0)[c]{\small $2$}}
\put(100,35){\makebox(0,0)[c]{\small $3$}}
\qbezier(80,10)(80,10)(90,15)
\qbezier(90,10)(90,10)(80,15)
\qbezier(100,10)(100,10)(100,15)
\put(85,12.5){\circle{1,4}}
\qbezier(80,15)(80,15)(90,20)
\qbezier(90,15)(90,15)(80,20)
\qbezier(100,15)(100,15)(100,20)
\qbezier(80,20)(80,20)(80,25)
\qbezier(90,20)(90,20)(100,25)
\qbezier(100,20)(100,20)(90,25)
\qbezier(80,25)(80,25)(80,30)
\qbezier(90,25)(90,25)(100,30)
\qbezier(100,25)(100,25)(90,30)
\put(95,27.5){\circle{1,4}}
\put(0,10){\circle*{1}}
\put(10,10){\circle*{1}}
\put(20,10){\circle*{1}}
\put(40,10){\circle*{1}}
\put(50,10){\circle*{1}}
\put(60,10){\circle*{1}}
\put(80,10){\circle*{1}}
\put(90,10){\circle*{1}}
\put(100,10){\circle*{1}}
\put(0,40){\circle*{1}}
\put(10,40){\circle*{1}}
\put(20,40){\circle*{1}}
\put(40,20){\circle*{1}}
\put(50,20){\circle*{1}}
\put(60,20){\circle*{1}}
\put(80,30){\circle*{1}}
\put(90,30){\circle*{1}}
\put(100,30){\circle*{1}}
\put(90,5){$c$}
\end{picture} \caption{Elements $a = \lambda_{2,3} \lambda_{1,3}, \, b = \lambda_{2,3}, \, c = \lambda_{2,3} \lambda^{-1}_{1,2}   \in FVB_3$.} \label{fig8}
\end{center} 
\end{figure}

\begin{theorem}\label{FVHP_3}
	We have the following decomposition
	$$
	X_3=\mathbb{Z}^2*\mathbb{F}_3*\Gamma,
	$$
	where $\mathbb{F}_3$ is a free group of rank $3$ and $\Gamma =\langle x, y, u, v, p, q~|~xy=uv, vu=pq, qp=yx\rangle $.
\end{theorem}

\begin{proof}
	Consider the restriction of the homomorphism $\nu_3\colon FVB_3\to S_3'$ to $FVP_3$. Let us denote it by $\varphi\colon FVP_3\to S_n'$. Then $X_3= \operatorname{Ker}(\varphi)$.
	
	To find the generators and relations of the $X_3$ group, we use the Reidemeisetr-Schreier rewriting process, see for example~\cite{MKS}.
	Let us write out the system of Schreier representatives for $\operatorname{Ker}(\varphi)$ using the generators indicated in the presentation (\ref{three_gen_FVP}): $T = \{1, a, ab, c, cb, b\}$. For an element $g$, we denote its representative in $T$ by $\overline{g}$. Then the kernel $\operatorname{Ker}(\varphi)$ is generated by the following elements:
	\begin{align*}
		a\cdot a\cdot (\overline{a^2})^{-1}=a^2c^{-1}=t, & \, & a\cdot c\cdot (\overline{ac})^ {-1}=ac=m, \\
		ab\cdot a\cdot (\overline{aba})^{-1}=abab^{-1}=v, & \, &ab\cdot b\cdot (\overline{ab^2})^{-1 }=ab^2a^{-1}=w, \\
		ab\cdot c\cdot (\overline{abc})^{-1}=abcb^{-1}c^{-1}=p, & \, &c\cdot a\cdot (\overline{ca}) ^{-1}=ca=r, \\
		c\cdot c\cdot (\overline{c^2})^{-1}=c^2a^{-1}=g, &\, &cb\cdot a\cdot (\overline{cba})^{ -1}=cbab^{-1}a^{-1}=q, \\
		cb\cdot b\cdot (\overline{cb^2})^{-1}=cb^2c^{-1}=h, & \, &cb\cdot c\cdot (\overline{cbc})^{ -1}=cbcb^{-1}=y, \\
		b\cdot a\cdot (\overline{ba})^{-1}=bab^{-1}c^{-1}=x, & \, &b\cdot b\cdot (\overline{b^2 })^{-1}=b^2=f, \\
		b\cdot c\cdot (\overline{bc})^{-1}=bcb^{-1}a^{-1}=u.& &
	\end{align*}
	Further, the relations $taca^{-1}c^{-1}t^{-1}$ for $t\in T$ must be rewritten in new generators. For example, for $t=ab$ we get:
	$$
	\begin{gathered}
		ab(aca^{-1}c^{-1})b^{-1}a^{-1} \, = \, vbca^{-1}c^{-1}b^{-1} a^{-1}\, = \, vuaba^{-1}c^{-1}b^{-1}a^{-1} \\
		=\, vuq^{-1}cbc^{-1}b^{-1}a^{-1} \, = \, vuq^{-1}p^{-1}.
	\end{gathered}
	$$
	The rest of the relations are found similarly. As a result, we get:
	\begin{align*}
		&m=r, \quad &&m=tg, \quad &&r=gt, \\
		&xy=uv, \quad &&vu=pq, \quad &&qp=yx.
	\end{align*}
	It is now clear that the elements $g, t$ generate $\mathbb{Z}^2$, the elements $w, h, f$ generate $\mathbb{F}_3$, and the group generated by the elements $x$ $y $, $u$, $v$, $p$ and $q$ we denote by $\Gamma$.
\end{proof}

\begin{lemma}~\cite{B23} \label{lemma2.6}
	Let $G_n$ be a subgroup of $FVP_n$ generated by the elements:
	\begin{align}
		t_{i,j}=&\lambda_{i,j}^{2}, \quad &&1\le i<j\le n, \label{def-t} \\
		d_{i,j,k}=&\lambda_{j,k}^{-1}\lambda_{i,j}^{-1}\lambda_{j,k} \lambda_{i,k}, \quad &&1\le i<j<k\le n,\\
		e_{i,j,k}=&\lambda_{j,k}^{-1}\lambda_{i,j}^{-1}\lambda_{i,k} \lambda_{i,j}, \quad &&1\le i<j<k\le n.
	\end{align}
	Then the normal closure of $G_n$ in $FVP_n$ coincides with $X_n$.
\end{lemma}

Let us describe the action of $\Theta_n^w$ on the generators indicated in Lemma~\ref{lemma2.6}.

\begin{lemma}\label{lemma3.3}
	The homomorphism $\Theta_n^w : FVB_n \to \operatorname{Aut} (\mathbb F_{2n})$, defined by the word $w \in \mathbb F_2$, maps the generators of the group $G_n$ to the following automorphisms:
	\begin{align}
		\Theta_n^w (t_{i,j})&:\begin{cases}
			x_i \mapsto x_i \, w_{i,j}^{-1} \, w_{j,i}^{-1},\\
			x_j\mapsto x_j \, w_{i,j} \, w_{j,i},
		\end{cases} \quad &&1 \le i<j\le n,\label{gamma}\\
		\Theta_n^w(d_{i,j,k})&:\begin{cases}
			x_j\mapsto x_j\, w_{j,i}^{-1}\, w_{i,k}^{-1} \, w_{j,k},\\
			x_k\mapsto x_k\, w_{i,k} \, w_{j,i} \, w_{j,k}^{-1},
		\end{cases} \quad &&1\le i<j<k\le n,\label{delta}\\
		\Theta_n^w (e_{i,j,k})&:\begin{cases}
			x_i\mapsto x_i\, w_{i,j}^{-1} \, w_{j,k}^{-1} \, w_{i,k},\\
			x_j\mapsto x_j\, w_{i,j} \, w_{i,k}^{-1} \, w_{j,k} ,
		\end{cases} \quad &&1\le i<j<k\le n.\label{eps}
	\end{align}
	where $w_{i,j} = w (y_i, y_j)$ for all $i,j$.
\end{lemma}

\begin{proof}
	Let, as before, $s_{i}=\Theta_n^w (\sigma_{i})$ and $r_{i}=\Theta_n^w(\rho_{i})$.
	First of all, let us establish some auxiliary formulas. Let $1\le i<j-1\le n-1$, then 
	\begin{align*}
		b_{i,j} :=\Theta^w_n(\rho_{j-1}\ldots\rho_{i+1}):\begin{cases}
			x_{i+1}\overset{r_{i+1}}{\longmapsto}x_{i+2},\\
			\vdots\\
			x_{j-1}\overset{r_{j-1}}{\longmapsto}x_j,\\
			x_j\overset{r_{j-1}}{\longmapsto}x_{j-1}\overset{r_{j-2}}{\longmapsto}x_{j-2}\ldots \overset{r_{i+1}}{\longmapsto}x_{i+1},\\
			y_{i+1}\overset{r_{i+1}}{\longmapsto}y_{i+2},\\
			\vdots\\
			y_{j-1}\overset{r_{j-1}}{\longmapsto}y_j,\\
			y_j\overset{r_{j-1}}{\longmapsto}x_{j-1}\overset{r_{j-2}}{\longmapsto}x_{j-2}\ldots \overset{r_{i+1}}{\longmapsto}y_{i+1}.
		\end{cases}
	\end{align*}
	Further,
	\begin{align*}
		a_{i,i+1} :=\Theta^w_n(\lambda_{i,i+1})=\Theta^w_n(\rho_i\sigma_i):\begin{cases}
			x_i\overset{r_{i}^{}}{\longmapsto}x_{i+1}\overset{s_{i}^{}}{\longmapsto}x_i \, w_{i,i+1}^{-1},\\
			x_{i+1}\overset{r_{i}^{}}{\longmapsto}x_i\overset{s_{i}^{}}{\longmapsto}x_{i+1} \, w_{i,i+1},\\
			y_i\overset{r_{i}^{}}{\longmapsto}y_{i+1}\overset{s_{i}^{}}{\longmapsto}y_{i+1},\\
			y_{i+1}\overset{r_{i}^{}}{\longmapsto}y_i\overset{s_{i}^{}}{\longmapsto}y_{i}.
		\end{cases}
	\end{align*}	
	Let us show that for $1\le i<j\le n$ the formulas
	\begin{align*}
		\Theta^w_n(\lambda_{i,j}):\begin{cases}
			x_i\mapsto x_i \, w_{i,j}^{-1},\\
			x_j\mapsto x_j \, w_{i,j},\\
			y_i\mapsto y_j,\\
			y_j\mapsto y_i,
		\end{cases}
		\qquad
		\Theta^w_n(\lambda_{i,j}^{-1}):\begin{cases}
			x_i\mapsto x_i \, w_{j,i},\\
			x_j\mapsto x_j \, w_{j,i}^{-1},\\
			y_i\mapsto y_j,\\
			y_j\mapsto y_i
		\end{cases}
	\end{align*} hold.
	Indeed, we have 
	\begin{align*}
		\Theta^w_n(\lambda_{i,j}):\begin{cases}
			x_i\overset{b_{i,j}}{\longmapsto}x_i\overset{a_{i,i+1}}{\longmapsto}x_i \, w_{i,i+1}^{-1}\overset{b_{i,j}^{-1}}{\longmapsto}x_i \,w_{i,j}^{-1},\\
			x_{i+1}\overset{b_{i,j}}{\longmapsto}x_{i+2}\overset{a_{i,i+1}}{\longmapsto}x_{i+2}\overset{b_{i,j}^{-1}}{\longmapsto}x_{i+1},\\
			\vdots\\
			x_{j-1}\overset{b_{i,j}}{\longmapsto}x_j\overset{a_{i,i+1}}{\longmapsto}x_j\overset{b_{i,j}^{-1}}{\longmapsto}x_{j-1},\\
			x_j\overset{b_{i,j}}{\longmapsto}x_{i+1}\overset{a_{i,i+1}}{\longmapsto}x_{i+1} \, w_{i,i+1}\overset{b_{i,j}^{-1}}{\longmapsto}x_j \, w_{i,j},\\
			y_i\overset{b_{i,j}}{\longmapsto}y_i\overset{a_{i,i+1}}{\longmapsto}y_{i+1}\overset{b_{i,j}^{-1}}{\longmapsto}y_j,\\
			y_{i+1}\overset{b_{i,j}}{\longmapsto}y_{i+2}\overset{a_{i,i+1}}{\longmapsto}y_{i+2}\overset{b_{i,j}^{-1}}{\longmapsto}y_{i+1},\\
			\vdots\\
			y_{j-1}\overset{b_{i,j}}{\longmapsto}y_j\overset{a_{i,i+1}}{\longmapsto}y_j\overset{b_{i,j}^{-1}}{\longmapsto}y_{j-1},\\
			y_j\overset{b_{i,j}}{\longmapsto}y_{i+1}\overset{a_{i,i+1}}{\longmapsto}y_i\overset{b_{i,j}^{-1}}{\longmapsto}y_i.
		\end{cases}
	\end{align*}
	We are now ready to prove the formulas (\ref{gamma})~-- (\ref{eps}). For example, let us establish (\ref{gamma}):
	\begin{align*}
		\Theta^w_n (t_{i,j}) = \Theta^w_n(\lambda_{i,j}^2):\begin{cases}
			x_i\overset{\lambda_{i,j}}{\longmapsto}x_i \, w_{i,j}^{-1}\overset{\lambda_{i,j}}{\longmapsto}x_i \, w_{i,j}^{-1} \, w_{j,i}^{-1},\\
			x_j\overset{\lambda_{i,j}}{\longmapsto} x_i \, w_{i,j}\overset{\lambda_{i,j}}{\longmapsto} x_i \, w_{i,j} \, w_{j,i},\\
			y_i\overset{\lambda_{i,j}}{\longmapsto} y_j\overset{\lambda_{i,j}}{\longmapsto} y_i,\\
			y_j\overset{\lambda_{i,j}}{\longmapsto} y_i\overset{\lambda_{i,j}}{\longmapsto} y_j.
		\end{cases}
	\end{align*} 
	The formulas (\ref{delta}) and (\ref{eps}) are proved in the same way.
\end{proof}

The following statement answers the question about the faithfulness of the representations $\Theta_n^{w,m}$ in the case $n=2$.
\begin{theorem} \label{theorem2.5}
	The nontrivial representation $\Theta_2^{w,m} : FVB_2\to \operatorname{Aut} ( \mathbb F_4)$ is not exact if the defining word $w$ is $$w(A,B)=A^{k_1} B^{k_2}\ldots A^{k_m}B^{-k_m}\ldots A^{-k_2}B^{-k_1}A^{m_1},$$ where all $k_i$ are nonzero integers except possibly only for $k_1$ and $k_m$. In this case, $\operatorname{Ker}(\Theta_2^{w,m})=X_2\simeq\mathbb{Z}$. The representation of $\Theta_2^{w,m}$ is exact for other $w$.
\end{theorem}

\begin{proof}
	Lemma \ref{lemma3.1} implies that ${\rm{Ker}}(\Theta_2^{w,m})\le X_2$. It is easy to show that $X_2$ is generated by the element $t_{1,2}$.
	
	In the case of $n=2$, the set $m$ consists of a single integer $m = \{m_1\}$. For any $k\in \mathbb{Z}$
	\begin{align*}
		\Theta_n^{w,m}(t_{1,2}^k)\colon \begin{cases}
			x_1\mapsto x_1 \, (w_{1,2}^{-1} \, y_2^{m_1} \, w_{2,1}^{-1} \, y_1^{m_1})^k,\\
			x_2\mapsto x_2 \, (w_{1,2} \, y_1^{-m_1} \, w_{2,1} \, y_2^{-m_1})^k.
		\end{cases}
	\end{align*}
	Thus $\Theta_n^w(t_{1,2}^k)={\rm id}$ iff either $k=0$ or $$w_{1,2}^{- 1} \, y_2^{m_1} \, w_{2,1}^{-1} \, y_1^{m_1}=1,$$ i.e. $f(y_1,y_2)=f^{-1} (y_2,y_1)$ for the word $f(y_1,y_2)=w(y_1,y_2)y_1^{-m_1}$.
	
	Let $f(A,B)=A^{k_1}B^{k_2} \dots B^{k_s}A^{k_{s+1}}$. Then
	$$
	A^{k_1}B^{k_2}\dots B^{k_s}A^{k_{s+1}}=B^{-k_{s+1}}A^{-k_s}\dots A^{ -k_2}B^{-k_1},
	$$
	therefore $f(A,B)=A^{k_1}B^{k_2} \dots A^{k_m}B^{-k_m}\ldots A^{-k_2}B^{-k_1}$, where all $k_i$ --- nonzero integers except maybe $k_1$ and $k_m$. But then $$w(A,B)=A^{k_1}B^{k_2}\ldots A^{k_m}B^{-k_m}\ldots A^{-k_2}B^{-k_1}A^ {m_1}.$$
	This completes the proof.
\end{proof}

\section{Nontriviality of the kernel ${\rm{Ker}}(\Theta_n^{w,m})$ for $n\ge3$} \label{section3}
Consider the following subgroups of the $FVB_n$ group:
\begin{align}
	Q_n^i&=\langle t_{i,i+1},~e_{i,i+1,i+2}\rangle, \quad &1\le i\le n-2,\\
	M_n^{i+1}&=\langle t_{i+1,i+2},~d_{i,i+1,i+2}\rangle, \quad &1\le i\le n-2, \\
	P_n^{i+2}&=\langle t_{i,i+2}\rangle, \quad &1\le i\le n-2.
\end{align}

\begin{lemma} \label{lemma4.1}
	Let $n\ge3$ and $w(A,B)\in \mathbb F_2 (A,B)$. Then, for all $i$, $1\le i \le n-2$, we obtain the inclusion
	\begin{align}
		\Big[Q_n^i,\big[M_n^{i+1},P_n^{i+2}\big]\Big] \le {\rm{Ker}}(\Theta_n^w). \label{incl}
	\end{align}
\end{lemma}

\begin{proof}
	Note that $Q_n^i$ acts non-trivially only on generators $x_i$ and $x_{i+1}$, $M_n^{i+1}$ acts non-trivially only on $x_{i+1}$ and $x_ {i+2}$, while $P_n^{i+2}$ acts non-trivially only on $x_i$ and $x_{i+2}$.
	Consider the element $$h=q(mpm^{-1}p^{-1})q^{-1}(mpm^{-1}p^{-1})^{-1},$$ where $q\in Q_n^i$, $m\in M_n^{i+1}$ and $p\in P_n^{i+2}$. This element acts non-trivially only on the generators $x_i$, $x_{i+1}$ and $x_{i+2}$. Write out its action:
	$$
	h : x_i\overset{q}{\longmapsto}qx_i\overset{m}{\longmapsto}qx_i\overset{p}{\longmapsto}pqx_i\overset{m^{-1}}{\longmapsto} 
	pqx_i\overset{p^{-1}}{\longmapsto} qx_i\overset{q^{-1}}{\longmapsto}x_i\overset{p}{\longmapsto} px_i\overset{m}{\longmapsto}px_i\overset{p^{-1}}{\longmapsto}x_i\overset{m^{-1}}{\longmapsto}x_i; 
	$$
	$$
	\begin{gathered}	
		h : x_{i+1}\overset{q}{\longmapsto}qx_{i+1}\overset{m}{\longmapsto}mqx_{i+1}\overset{p}{\longmapsto}mqx_{i+1}\overset{m^{-1}}{\longmapsto}qx_{i+1}\overset{p^{-1}}{\longmapsto} qx_{i+1}\overset{q^{-1}}{\longmapsto}x_{i+1}\overset{p}{\longmapsto}  x_{i+1} \cr \overset{m}{\longmapsto}mx_{i+1} \overset{p^{-1}}{\longmapsto}mx_{i+1}\overset{m^{-1}}{\longmapsto}x_{i+1};
	\end{gathered}
	$$
	$$
	\begin{gathered}
		h : x_{i+2}\overset{q}{\longmapsto}x_{i+2}\overset{m}{\longmapsto}mx_{i+2}\overset{p}{\longmapsto}pmx_{i+2}\overset{m^{-1}}{\longmapsto}m^{-1}pmx_{i+2}\overset{p^{-1}}{\longmapsto}  p^{-1}m^{-1}pmx_{i+2}\overset{q^{-1}}{\longmapsto} \cr p^{-1}m^{-1}pmx_{i+2}\overset{pmp^{-1}m^{-1}}{\longmapsto}x_{i+2}
	\end{gathered}
	$$
	Thus, $h\in{\rm{Ker}}(\Theta_n^w)$ and the inclusion (\ref{incl}) is proved.
\end{proof}

\begin{theorem} \label{theorem3.2}
	Let $n \geq 3$. For any defining word $w(A,B)\in \mathbb F_2 (A,B)$ the kernel ${\rm{Ker}}(\Theta_n^w)$ contains a subgroup isomorphic to a free group of rank 2.
\end{theorem}

\begin{proof}
	Denote a free group of rank 2 by $\mathbb F_2$. By Lemma~\ref{lemma4.1}, it suffices to show that $\mathbb F_2 \le \Big[Q_3^1,\big[M_3^2,P_3^3\big]\Big]\le FVP_3$. 
	
	Consider the elements
	$$
	h_0 =\big[t_{1,2},[t_{2,3},t_{1,3}]\big] \qquad \text{and} \qquad h_1 =\big[t_{1,2},[d_{1,2,3},t_{1,3}]\big].
	$$
	We write them down in terms of the generators of $FVP_3$ group, see (\ref{three_gen_FVP}):
	\begin{equation}
		h_0 =\big[(c^{-1}b)^2,[b^2,(b^{-1}a)^2]\big] \qquad \text{and} \qquad h_1 =\big[(c^{-1}b)^2,[b^{-2}ca,(b^{-1}a)^2]\big].
	\end{equation}
	
	Let us prove that $h_0$ and $h_1$ generate $\mathbb F_2$. To do this, it suffices to show that there are no relations between them. Let $\psi : FVP_3 \to \langle a, b \rangle$ be the homomorphism given by the mapping $\psi(a) = a$, $\psi(b) = b$ and $\psi (c) = 1 $. Denote $\bar{h}_0 = \psi (h_0)$ and $\bar{h}_1 = \psi (h_1)$. Then
	\begin{align*}
		\bar{h}_0&=b^3ab^{-1}ab^{-2}a^{-1}ba^{-1}b^{-2}ab^{-1}ab^2a^{-1}ba^{-1}b^{-1}, \\
		\bar{h}_1&=ab^{-1}aba^{-1}ba^{-1}b^{-2}ab^{-1}ab^{-1}a^{-1}ba^{-1}b^2.
	\end{align*}
	The elements $\bar{h}_0$ and $\bar{h}_1$ lie in the free group $\langle a, b \rangle$. Hence the group $\langle \bar{h}_0,\bar{h}_1\rangle$ is either isomorphic to $\mathbb{Z}$ or isomorphic to $\mathbb F_2$. The first case means that $\bar{h}_0$ and $\bar{h}_1$ must be powers of the same element, i.e. $\bar{h}_0=g(a,b)^k$ and $\bar{h}_1=g(a,b)^s$ for some word $g(a,b)\in\langle a , b \rangle$ and nonzero $k,s\in\mathbb{Z}$. Let $g(a,b)=f\cdot w\cdot f^{-1}$, where $w(a,b)$ is the cyclic reduced word in $\langle a, b \rangle$. Then $g^s=fw^sf^{-1}=\bar{h}_1$ and since $\bar{h}_1$ is itself cyclically reduced, we get $f=1$. But then $g^k=fw^kf^{-1}=w^k=\bar{h}_0$ must be cyclically reduced, which is not the case. Thus $\langle \bar{h}_0,\bar{h}_1\rangle \cong \mathbb F_2$ and hence $\langle h_0,h_1\rangle \cong \mathbb F_2$.
\end{proof}

\begin{corollary}\label{corollaryLast}
	Let $n \geq 3$, then $\operatorname{Ker}(\Theta_n^{w,m})$ contains a subgroup isomorphic to a free group of rank 2 for any integer tuple $m = (m_1, \ldots, m_{n-1})$ and arbitrary defining word $w(A,B)\in \mathbb F_2 (A,B)$.
\end{corollary}

\begin{proof}
	There are three points which are sufficient to prove the corollary:
	\begin{itemize}
		\item $t_{1,2}$ acts non-trivially only on generators $x_1$ and $x_2$, 
		\item $t_{2,3}$ and $d_{1,2,3}$ act non-trivially only on elements $x_2$ and $x_3$, 
		\item $t_{1,3}$ acts non-trivially only on $x_1$ and $x_3$. 
	\end{itemize}
	
	Recall we have $t_{i,i+1} = \lambda_{i,i+1}^2 = (\rho_i \sigma_i)^2$ by the formula (\ref{def-t}). Therefore, $t_{1,2}$ and $t_{2,3}$ satisfy property above.
	
	Let's check that $t_{1,3}=(\rho_2 \rho_1 \sigma_1\rho_2)^2$ leaves $x_2,~y_1,~y_2$ and $y_3$ in place. The element $t_{1,3}$ acts trivially on the generators $y_i$ for $1\le i\le 3$, because using formulas~(\ref{homodef1}, \ref{homodef2}) we get $t_{1,3}\cdot y_i=(\rho_2 \rho_1 \rho_2)^2\cdot y_i = 1\cdot y_i=y_i$. The trivial action on $x_2$ follows from the fact that $\rho_2 \rho_1 \sigma_1\rho_2$ leaves this element in place. Indeed, according to~(\ref{homodef1}, \ref{homodef2}), we obtain
	$$
	\begin{gathered}
		x_2\overset{\rho_2}{\longmapsto} x_3 y_3^{m_2} \overset{\rho_1}{\longmapsto} x_3 y_3^{m_2}\overset{\sigma_1}{\longmapsto} x_3y_3^{m_2}\overset{\rho_2}{\longmapsto} x_2 y_2^{-m_2} y_2^{m_2} = x_2.
	\end{gathered}
	$$
	Similarly, we check the action for
	$$
	\begin{gathered}
		x_1\overset{\sigma_2\rho_2}{\longmapsto} x_1 \overset{\sigma_1}{\longmapsto} x_2w(y_1,y_2)\overset{\rho_1}{\longmapsto} x_1y_1^{-m_1}w(y_2,y_1)\overset{\rho_2}{\longmapsto} x_1y_1^{-m_1}w(y_3,y_1)\overset{\sigma_2}{\longmapsto}\\
		\overset{\sigma_2}{\longmapsto}
		x_1y_1^{-m_1}w(y_3,y_1)\overset{\rho_2}{\longmapsto}
		x_1y_1^{-m_1}w(y_2,y_1)\overset{\rho_1}{\longmapsto}\\
		\overset{\rho_1}{\longmapsto}
		x_2y_2^{m_1}y_2^{-m_1}w(y_1,y_2)\overset{\sigma_1}{\longmapsto} x_1w^{-1}(y_1,y_2)w(y_1,y_2)\overset{\rho_2}{\longmapsto} x_1,\\
		d_{1,2,3}\cdot y_i=\rho_1\rho_2\rho_1\rho_1\rho_2\rho_1 \cdot y_i=y_i,\quad 1\le i\le 3.
	\end{gathered}
	$$
	This completes the proof.
\end{proof}

\end{document}